\documentclass{amsart}
\usepackage{amsmath,amssymb,amsthm,amscd}
\usepackage{color}
\newtheorem{thm}{Theorem}

\newtheorem{cor}{Corollary}
\newtheorem{prop}{Proposition}
\theoremstyle{definition}

\theoremstyle{remark}

\newtheorem{defin}{Definition}

\begin{document}
\title{A characterization of Lelong classes on toric manifolds}

\author{Maritza M. Branker}
\address{Department of Mathematics, Niagara University, NY 14109}
\email{mbranker@niagara.edu}
\thanks{}
\author{Ma{\l}gorzata Stawiska}
\address{Department of Mathematics, University of Kansas, Lawrence, KS 66045}
\email{stawiska@ku.edu}
\thanks{Preliminary version}
\dedicatory{} 

\maketitle

\section{Introduction}

There are many connections between pluripotential theory in $\mathbb{C}^n$ (both the classical theory and weighted analog) and the study of complex polynomials. For example,
the following theorem, due to Siciak (\cite{Si}; see also \cite{Kl}, Theorem 5.1.6),  provides a polynomial approximation of the Lelong class $\mathcal{L}$ of plurisubharmonic (psh) functions in $\mathbb{C}^n$ of logarithmic growth. It also shows a relation between $\mathcal{L}$ and the class $\mathcal{H}^n_+$ of nonnegative psh functions in $\mathbb{C}^n$ which are absolutely homogeneous of order one.  \\

\begin{thm}\label{thm:Lclass} Let $h: \mathbb{C}^n \mapsto [0,\infty), \; u: \mathbb{C}^n \mapsto [-\infty,\infty)$ be functions such that $h \not \equiv 0$ and $u \not \equiv -\infty$.\\
(i) If $h \in \mathcal{C}(\mathbb{C}^n) \cap \mathcal{H}^n_+$ and $h^{-1}(0)=\{0\}$, then $h(z)=\sup \{|Q(z)|^{1/(\mbox{deg }Q)}\}, z \in \mathbb{C}^n$, where the $\sup$ is taken over all complex homogeneous polynomials $Q$ such that $|Q|^{1/(\mbox{deg }Q)}\leq h$ in $\mathbb{C}^n$.\\
(ii) A sufficient and necessary condition for $h$ to belong to $\mathcal{H}^n_+$ is that 
\newline $h = (\limsup_{j \to \infty}|Q_j|^{1/j})^*$ for some sequence of complex homogeneous polynomials such that $\mbox{deg }Q_j \leq j$. In particular, if $h \in  \mathcal{H}^n_+$, then $\log h \in \mathcal{L}$.\\
(iii) The function $u$ belongs to $\mathcal{L}$ if and only if
$e^u=(\limsup_{j \to \infty}|P_j|^{1/j})^*$ for some sequence of
complex polynomials on $\mathbb{C}^n$ such that $\mbox{deg }P_j
\leq j$.
\end{thm}

Looking for possible generalizations of this theorem, one should take into account the following (cf. \cite{BS}, Proposition 5 and Theorem 1, and the references given there): considering $\mathbb{C}^n$ as the complement of the hyperplane at infinity in $\mathbb{CP}^n$, the class $\mathcal{L}(\mathbb{C}^n)$ corresponds in the 1-to-1 manner with the class $PSH(\mathbb{CP}^n,\omega_{FS})$ of functions quasi-plurisubharmonic with respect to the Fubini-Study form $\omega_{FS}$. There is a also a one to one correspondence the class of positive singular metrics on the hyperplane bundle $L$ over $\mathbb{CP}^n$. By Grauert's characterization of positive line bundles, all of these correspond with the class of plurisubharmonic functions in the total space of the universal bundle $L'$ over $\mathbb{CP}^n$ (dual to the hyperplane bundle) which are nonnegative, not identically zero and absolutely homogeneous of order one in each fiber. 
Furthermore, homogeneous polynomials of degree $d \geq 1$ in $L'$ can be thought of as sections of the $d$-th tensor power $dL$ (we use the additive notation for the tensor product operation).   The notion of homogeneous polynomials on  the total space of $L'$ and the correspondence between them and  holomorphic sections of $L$ is  also  valid for certain line bundles over  toric manifolds (see below for precise statements).  We will work in this setting to prove an analog of Theorem \ref{thm:Lclass} for  a toric variety $X$ with a line bundle $\mathcal{O}_X(D)$ coming from a Cartier divisor $D$ with a strictly convex support function (see below for relevant definitions and properties). As an application of this theorem we will prove an approximation theorem for Chern classes of such line bundles.

\section{Line bundles over toric varieties}

Let's begin by recalling some facts about toric varieties and line bundles. A toric variety $X_\Sigma$ can be constructed from a fan $\Sigma$ in $\mathbb{R}^n$,that is a certain collection of cones generated by vectors from $\mathbb{Z}^n$.
Let $v_1,...,v_N$ denote the primitive generators of one-dimensional cones in $\Sigma$. A character on $\mathbb{Z}^N$ (i.e., an integer-valued function
$\varphi: \mathbb{Z}^N \mapsto \mathbb{Z}_+$) is determined by its values $a_i=\varphi(e_i)$, where $e_i, \; i=1,...,N$ are the vectors of the standard basis, and it defines a function $\varphi: \mathbb{R}^n \rightarrow \mathbb{R}$  by putting $\varphi(v_i)=-a_i$ and  extending it as a linear function on each cone $\sigma \in \Sigma$. We will follow the convention from \cite{CLS}; according to their definition 6.1.17, a function $\varphi:S \mapsto \mathbb{R}$ defined on a convex set $S \subset \mathbb{R}^n$ is convex if and only if $\varphi(tu+(1-t)tv) \geq t\varphi(u)+(1-t)\varphi(v)$ for all $u,v \in S$ and $t \in [0,1]$.
The support function $\varphi_D$ of a Cartier divisor on $X_\Sigma$ ( with the Cartier data $\{m_\sigma: \sigma \in \Sigma(n)\}$) is called strictly convex if it is convex and for every cone $\sigma \in \Sigma(n)$ it satisfies $<m_\sigma,u>=\varphi_D(u) \Leftrightarrow u \in \sigma$.\\
The function $\varphi=\varphi_D$ determines a torus invariant divisor $D$ on $X_\Sigma$, $D=\sum_{i=1}^N a_iD_i$, where $D_i$ are hypersurfaces corresponding to the one-dimensional cones in $\Sigma$. By Propositions 4.3.3 and 4.3.8 in \cite{CLS} (see also Proposition 2 in \cite{Ro}), the line bundle associated with $D$ has global sections $\Gamma(X_\Sigma,\mathcal{O}_{X_\Sigma}(D))=\bigoplus_{m \in P_D \cap \mathbb{Z}^n}\mathbb{C}\cdot \chi^m$, where $P_D \subset \mathbb{R}^n$ is the polyhedron defined by $P_D=\{m \in \mathbb{R}^n: <m,v_i> \; \geq -a_i, \quad i=1,...,N\}$. By Lemma 6.1.9 in \cite{CLS}, $P_D=\{m\in \mathbb{R}^n: \varphi(x) \leq \; <m,v> \forall v \in |\Sigma|\}$. By Theorem 6.1.10 in \cite{CLS}, $\varphi_D$ is convex if and only if $D$ is basepoint-free (i.e., $\mathcal{O}_{X_\Sigma}(D)$ is generated by global sections).  This in turn is equivalent to the property that $\varphi_D(u)=\min_{m \in P_D}<m,u> \quad \forall u \in \mathbb{R}^n$. Finally by Theorem 6.1.15 in \cite{CLS}, $D$ is ample if and only if $\varphi_D$ is strictly convex.\\

 From now on we assume that $X$ is a compact toric variety with a line bundle $L= \mathcal{O}_X(D)$ coming from a Cartier divisor $D$ with a strictly convex support function (hence $L$ is ample). Recall Grauert's ampleness criterion for line bundles (\cite{MM}, Theorem B.3.13 and the references there) which says that a holomorphic line bundle $L$ over a compact complex manifold $X$ is ample if and only if the zero section $Z(L')$ of the dual bundle $L'$ has a strongly pseudoconvex neighborhood. 	It follows that $Z(L')$ can be blown down to a finite set of points.\\

Consider the following cone in $\mathbb{R}^{n+1}$: $C_\varphi=\{(x,\alpha)\in \mathbb{R}^n\times \mathbb{R}: \alpha \geq \varphi(x)\}$. This is equal to $\mbox{Cone}\{(0,1),(v_i,-a_i):i=1,...,N\}$. By \cite{Ro}, Proposition 1 (cf. also Exercise 6.1.13 in \cite{CLS}), the function 
 $\varphi$ is strictly convex if and only if $C_\varphi$ is a convex cone in $\mathbb{R}^{n+1}$ with all its proper faces simplicial and with $(v_i,-a_i): i=1,...,N$ as the fundamental generators. \\
%In this situation the cone $C_\varphi$ is strongly convex, i.e., $C_\varphi \cap (-C_\varphi) =\{0\}$.\\

For $\sigma \in \Sigma$ let $\tilde{\sigma}$ be $\{(u,\alpha): u\in \sigma, \alpha \geq \varphi(u)\}= \mbox{Cone }\{(0,1),(v_i,-a_i):v_i \in \sigma(1)\}, \quad i=1,...,N$.
By proposition 7.3.1 in \cite{CLS}, the affine variety $U_{\tilde{\sigma}}$ is isomorphic to $U_\sigma \times \mathbb{C}$, which is a trivializing neighborhood for the line bundle $\mathcal{O}_{X_\Sigma}(D)$. The projection induces a map $U_{\tilde{\sigma}} \mapsto \mathbb{C}$, which is $\chi^{(-m_\sigma,1)}$.
This suggests another way of looking at the sections of $\mathcal{O}_{X_\Sigma}(D)$. Specifically, the following characterization appears in \cite{Ro} (discussion between Lemma 5 and Proposition 3):\\
Let $A_\varphi$ be the affine toric variety (with the action of the torus $T^{n+1}$) associated with
the cone $C_\varphi$: $A_\varphi = \mbox{Spec }\mathbb{C}[\mathbb{Z}^{n+1} \cap C_\varphi^*]$. This is the affine variety obtained by blowing down the zero section of the line bundle $\mathcal{O}_{X_\Sigma}(-D)$. The affine coordinates on $A_\varphi$ are given by the characters $\chi^m$ such that $m \in P_D \cap \mathbb{Z}^n$. One has thus the identification  $\Gamma(X_\Sigma,\mathcal{O}_{X_\Sigma}(D))=\{A_\varphi \mapsto
\mathbb{C}: f(\lambda \cdot z)=\lambda f(z), \lambda \in \mathbb{C}^*, z \in A_\varphi\}$, so the sections of $L$ can be thought of as linear forms on $A_\varphi$. This kind of identification can be extended to $\Gamma(X,dL)$. In fact, the following holds:

\begin{prop} (\cite{Ba}; \cite{Co}, Section 5; cf. also \cite{CLS}, Appendix A to the Chapter 7): Let $X$ and $L$ be as above. Consider the graded ring
$R_D = \bigoplus_{d \in \mathbb{N}}\Gamma(X,dL)$. Then there exists a ring isomorphism between $R_D$ and $\mathbb{C}[\mathbb{Z}^{n+1} \cap C_\varphi^*]$, preserving the grading.
\end{prop}

The grading  in this proposition is given by $\mbox{deg }t_0^d\chi^m(t)=d, d \geq 1$. A closely related result, proved using the above isomorphism, is the following: if one considers $R_D$ as a $\mathbb{C}$-algebra with the operation of multiplication of
complex functions, then $R_D$ is finitely generated \cite{El}.We can now formulate our first result:

\begin{thm}
\label{thm:hom-psh} Let  $X$ be a toric manifold and $L$ be a line bundle over $X$ defined by a Cartier divisor $D$ with a strictly convex support function.  Let $H: L' \mapsto [0,\infty), \quad H \not \equiv 0$ be a  continuous plurisubharmonic function which is homogeneous in each fiber, with $H^{-1}(0)=Z_{L'}$, where $Z_{L'}$ is the zero section of $L'$. Then $H(z)=\sup \{|Q(z)|^{1/(\mbox{deg }Q)}\}$, where $Q$ is a combination of the variables $t_0^d\chi^m(t)$ (a homogeneous polynomial in coordinates of $A_\varphi$),  $\mbox{deg }Q = d, \quad m \in dP_D$, such that $|Q|^{1/(\mbox{deg }Q)}\leq H$.\\
\end{thm}

\begin{proof} Consider $S=\bigcup_{z \in Z_{L'}}\{(z,t): t \in L'_z: |t|=1\}$. Then $M=\inf_S H$ is positive. By homogeneity in each fiber, $H(z,t) \leq M|t|$. Let $a \in L'$ be such that $H(a)=1$ and let $g(a)$ denote the supremum on the right-hand side evaluated at $a$. It is enough to show that $g(a) \geq 1$. 
The set $\Omega=\{z: H(z) < 1\}$ is a bounded pseudoconvex neighborhood  of $Z_{L'}$. 
After blowing down the zero section, we get a holomorphically convex bounded neighborhood $\tilde{\Omega}$ of a point in $A_\varphi$, hence a polynomially convex set.  Hence for any $\lambda \in (0,1)$ we have $\hat{K_\lambda} \subset \tilde{\Omega}$, where $K_\lambda=H^{-1}([0,\lambda])$. Now recall that the polynomial ring $\mathbb{C}[\mathbb{Z}^{n+1} \cap C_\varphi^*]$ is the subring of $\mathbb{C}[t_0,t_1^{\pm 1},...,t_n^{\pm 1}]$ spanned by Laurent monomials
$t_0^d\chi^m(t)$ with $m \in dP_D$ and the sections of $dL$ correspond exactly to homogeneous polynomials which are sums of $t_0^d\chi^m(t)$.  
Recall also that the fiber $L^j_z$ can be identified with the space of $j$-linear functions $L'_z \times ...\times L'_z \mapsto \mathbb{C}$. In particular, every section in $L^j$ is a homogeneous function of degree $j$ on every fiber of $L'$.\\

Fix a $\lambda \in (0,1)$. There exists a $\mu \in (\lambda,1)$ and a homogeneous polynomial $Q$ (viewed as a homogeneous function of $A_\varphi$) such that $Q(\mu a) \geq 1$ and $\|Q\|_{K_\lambda} \leq 1$. By homogeneity, $\lambda |Q|^{1/\mbox{deg }Q}$, so $\lambda \leq |\lambda^{\mbox{deg }Q}Q(\mu a)|^{1/\mbox{deg }Q} \leq g(\mu a) =\mu g(a)$. Letting $\lambda \to 1$ implies that $\mu \to 1$ and $1 \leq g(a)$. \end{proof}

We also have the following:
\begin{thm}\label{thm: limsup} Under the assumptions of Theorem \ref{thm:hom-psh}, $H = (\limsup_{j \to \infty}|Q_j|^{1/j})^*$ for some sequence of $Q_j$ of homogeneous polynomials on $A_\varphi$  with  $\mbox{deg }Q_j \leq j$.
\end{thm}

\begin{proof} 

 Let $H: L' \mapsto [0,\infty), \; H \not \equiv 0$ be a  continuous plurisubharmonic function which is homogeneous in each fiber, with $H^{-1}(0)=Z_{L'}$, where $Z_{L'}$ is the zero section of $L'$. The set $\Omega=\{z: H(z) < 1\}$ is a balanced pseudoconvex neighbourhood of $Z_{L'}$. We let $\Phi$ denote the map which blows down the zero section of $L'$ and $\tilde{\Omega} $ denote the set $\Omega$ maps to.

Since $\tilde{\Omega}$ is a holomorphically convex, there exists  a holomorphic function  $\tilde{F}$  on $\tilde{\Omega}$ which does not extend  beyond $\tilde{\Omega}$. Observe that $\tilde{\Omega}$ is a balanced neighbourhood of 0 in $A_{\varphi}$ so $\tilde{F}$ has a homogeneous expansion. That is, $\tilde{F} (z) = \sum_{j=0}^{\infty} Q_j(z) $ for $z \in \tilde{\Omega}$ where $Q_j$ denotes a homogeneous polynomial with  $\mbox{ deg } Q_j \leq j$. Define $v(z)= (\limsup_{j \to \infty}|Q_j(z)|^{1/j})^*$ for $z \in A_{\varphi}$.
By Cauchy's convergence criterion for numerical series $(\limsup_{j \to \infty}|Q_j(z)|^{1/j})^* \leq 1$ for $z \in \tilde{\Omega}$. Consequently $\tilde{\Omega} = \{ z\in A_{\varphi} : v(z) <1 \}$. By the homogeneity of $v$ and $H$ we have $v \equiv H$ on $A_\varphi$.\end{proof}

It is quite straightforward to show that the existence of a plurisubharmonic function $H:L' \mapsto [0,\infty)$ as in Theorem \ref{thm:hom-psh} is equivalent to the existence
of a positive singular hermitian metric on $L$, i.e., to $L$ being pseudoeffective. Namely (\cite{BS}, Theorem 1 and the references there), given such an $H$ and a system of trivializations $\theta_i: L'\mid_{U_i} \mapsto U_i \times \mathbb{C}$, the functions $h_i(x)= \log (H\circ \theta_i^{-1}(x,t)/|t|), \quad x \in U_i, t \neq 0$ give a positive singular metric $h^L$ on $L$, and the same formula defines $H$ as in Theorem \ref{thm:hom-psh} if the collection $\{h_i\}$ gives a positive singular metric. As an example of a positive singular metric one can take $\log|s|$, where $s$ is a holomorphic section of $L$.  Our theorems yield the following corollaries for positive singular metrics:\\

\begin{cor}\label{cor:hom-sing} Let a toric manifold $X$ and a line bundle $L$ be as above. Let $h$ be a positive singular metric on $L$. Then (in each triviaizing neighborhood) $\exp h(z)=\sup \{|Q(z)|^{1/(\mbox{deg }Q)}\}$, where $Q \in R_D$ are such that $|Q|^{1/(\mbox{deg }Q)}\leq h$.\\
\end{cor}
\begin{cor}\label{cor:limsup-sing} Under the assumptions of Theorem \ref{thm:hom-psh}, $\exp h = (\limsup_{j \to \infty}|Q_j|^{1/j})^*$ for some sequence of $Q_j \in R_D$ such that $Q_j \in \Gamma(X,jL)$.
\end{cor}

\section{Quasi-plurisubharmonic functions}

From \cite{GZ} we have the following relation between positive singular metrics and quasi-plurisubharmonic functions:
Consider the metric with the weight $\lambda$ (or fix any smooth metric on the line bundle $L$ over $X$) ,
and set $\omega = dd^c \lambda$. The class $PSH(X,\omega)$ of functions $u \in L^1(X,\mathbb{R}\cup \{-\infty\}$ such that $u$
is upper semicontinuous and satisfies $dd^c u \geq -\omega$ can be obtained by taking all positive singular metrics $\psi$ on  $L$
and setting $u = \psi - \lambda$. Conversely, if $u \in PSH(X,\omega)$, then $u+\lambda$ defines a positive singular metric on $L$.
We call $PSH(X,\omega)$ the Lelong class on $X$.

Our next task is to provide a characterization of positive singular metrics \newline
and quasi-plurisubharmonic functions on toric varieties.
We will use the properties of the integer-valued function $\varphi_D$ and the polyhedron $P_D$ associated with the Cartier divisor $D$ on $X$:
$P_D=\{m \in (\mathbb{Z}^n)^{\breve{}}=M: <m,v_i> \geq -a_i \forall i \}=\{m \in M: m \geq \varphi_D \}$.

In what follows we will assume that  $\mathcal{O}_X(D)$ is very ample or, equivalently, that $P=P_D$ is a very ample full-dimensional lattice polyhedron ($\mbox{dim }P=n$). (This term roughly means that $P$ contains enough lattice points; see Definition 7.1.8 in \cite{CLS} for its precise definition and Proposition 6.1.4 for equivalence of  the line bundle and the associated polytope being very ample. The conditions of being ample and being very ample are equivalent for torus invariant line bundles over toric manifolds. )  If  $P \cap M=\{m_1,...,m_r\}$, then the toric variety $X$ is the Zariski closure of the map $\Phi: T^n  \ni t \mapsto (\chi^{m_1}(t),...,\chi^{m_r}(t)) \in \mathbb{CP}^{r-1}$. Let $x_1,...,x_r$ denote the homogeneous coordinates in $\mathbb{CP}^{r-1}$ and let $U_i = \{x_i \neq 0\}$. According to \cite{CLS}, Theorem 2.3.1, for each vertex $m_i$ in $P \cap M$, 
the affine piece $X \cap U_i$ is the affine toric variety $U_{\sigma_i}=\mbox{Spec }\mathbb{C}[\breve{\sigma_i} \cap M]$, where $\sigma_i$ is the strongly convex rational polyhedral cone dual to the cone $\mbox{Cone }(P\cap M -m_i)$ and $\mbox{dim }\sigma_i=n$. The  Cartier data $m_\sigma$ are also  related to the vertices of $P$. Namely, by Theorem 6.1.10 d,  $D$ is basepoint-free if and only if   $\{m_\sigma\}$ is the set of vertices of $P$. Hence, if $D$ is very ample,  then for every maximal cone $\sigma \in \Sigma$ the semigroup $\breve{\sigma} \cap M$ is generated by $\{m-m_\sigma: m \in P_D \cap M\}$.

Note that
$\lambda =\log (\sum_{m \in  P_D \cap M}|\chi^m(z)|^2)^{1/2}$ is the pullback of the Fubini-Study metric
from $\mathbb{CP}^{r-1}$ by $\Phi$. We will use this metric to define the Lelong class $PSH(X,\omega)$ on $X$.

\begin{prop}
 Let $X$ be a smooth toric variety as in Theorem \ref{thm:hom-psh}. Then $h$ is (a weight of) a positive singular metric  if and only if on each affine chart $U_\sigma$ we have
$e^u=(\limsup_{j \to \infty}|P_j|^{1/j})^*$ for some sequence of
complex polynomials in the variables $\chi^{m-m_\sigma}, m \in M$ such that $\mbox{deg }P_j
\leq j$.
\end{prop}
\begin{proof} On the affine open set $U_\sigma$ associated with a maximal cone $\sigma$ one has 
\[
\Gamma(U_\sigma,\mathcal{O}_X)= \mathbb{C}[\breve{\sigma_v}\cap M].
\] In terms of characters on the torus, the last set coincides with complex polynomials in the variables $\chi^{m-m_\sigma}, m \in M$. The result now follows from Theorem \ref{thm: limsup}. 
\end{proof}

In analogy to the case of $\mathbb{CP}^n$, a growth condition for plurisubharmonic functions on a complex torus was introduced  in \cite{Be}, section 4. It generalizes the growth condition for the Lelong class in $\mathbb{C}^n$ for the purpose of studying quasi-plurisubharmonic functions on toric manifolds. Moreover, it is equivalent to the statement that an $\omega$-quasi-psh function $v$  on a toric manifold $X$ corresponds uniquely to a plurisubharmonic function $u$ on the torus $T^n$ satisfying $u(z_1,...,z_n) \leq \psi(z_1,...,z_n)+ \mathcal{O}(1):= \sup_{m \in P_D} <m,(\log|z_1|,...,\log |z_n|)> + \mathcal{O}(1)$. (Note that $\psi$ is the composition of the support function of $P_D$ with the map $z \mapsto (\log|z_1|,...,\log |z_n|)$.) Accordingly, $u -\psi + \lambda$ is a positive singular metric. We will now use our approximation results from the previous section to prove that this condition provides unique characterization of $\omega$-quasi-psh functions.

\begin{thm} Let  $X=X_\Sigma$ be a $n$-dimensional toric manifold with a complete fan $\Sigma$ and let $D$ be a very ample divisor on $X$. Let  a plurisubharmonic function $u$ on the torus $T^n$  be given.Then  $v = u(z_1,...,z_n)-\psi(\log|z_1|,...,\log |z_n|)$ extends to an $\omega$-psh function on $X$ if and only if $u$ satisfies $u(z_1,...,z_n) \leq \psi(\log|z_1|,...,\log |z_n|) + \mathcal{O}(1):= \sup_{m \in P_D} <m,(\log|z_1|,...,\log |z_n|)> + \mathcal{O}(1)$.
\end{thm}
\begin{proof} 
 If $u$ is a psh function satisfying the growth condtion, then $v$ and its extension to $X$ satisfy $dd^c v \geq dd^c (-\psi) =-[D]$. Conversely, Corollary  \ref{cor:limsup-sing}, $u-\psi + \lambda=  (\limsup_{j \to \infty}(1/j)\log |Q_j|)^*$ for some sequence $Q_j \in \Gamma(X, \mathcal{O}_X(jD)), \quad j=1,2,...$. We have $(1/j)\log |Q_j|)-\lambda \leq \mathcal{O}(1)$ for all $j$, and the inequality persists under the limit superior and upper semicontinuous regularization, so $u-\psi \leq \mathcal{O}(1)$.
\end{proof}

\section{Application: approximation theorem for curvature currents}

In this section we will apply Theorem \ref{thm: limsup} to prove an approximation theorem for first Chern classes of ample line bundles  over projective toric manifolds. 
Recall the following:
\begin{defin} (cf. \cite{De2}, Section 6) Let $L$ be a holomorphic line bundle over a compact complex manifold $X$.\\
(a) We say that $L$ is effective if $dL$ has a section for some $d > 0$, i.e., $\mathcal{O}(dL)=\mathcal{O}(D)$ for some effective divisor $D$ on $X$;\\

(b) The convex cone generated by $c_1(L)$ in $H^2(X,R)$ for $L$ effective is denoted by $N_{eff}$;  if $L$ is ample, the corresponding convex cone is denoted by $N_{amp}$. 

\end{defin}
Recall also that the first Chern class of a line bundle $L$ is its Euler class, i.e., the Poincar\`{e} dual to the homology class of the zero locus of a generic smooth section.\\

Using the results of the previous section we will now prove that over a projective toric manifold $X$
 we have $N_{amp} \subset \overline{N_{eff}}$. This result can be derived from (\cite{De2}, Proposition 6.6) for an arbitrary projective algebraic manifold, not necessarily toric. In a series of inclusions and equalities, one has  
$N_{amp} \subset N_{nef} \subset N_{psef}$ and then $N_{psef} =  \overline{N_{eff}}$, where $N_{nef}$ and $N_{psef}$ denote the convex cones generated by first Chern classes of respectively numerically effective and peudoeffective line bundles. We refer to Demailly's article \cite{De2} for definition of these notions, since we wil not use them.  (Note also that on toric varieties built from convex fans all nef line bundles are ample; this is the toric Nakai criterion, see e.g. \cite{Mu}, Theorem 3.2). We will also not use the equality  $N_{psef} =  \overline{N_{eff}}$ or Lelong numbers, which play central role in its proof. Instead, we will prove directly  the inclusion $N_{amp} \subset \overline{N_{eff}}$ on a toric projective manifold, relying only on Theorem \ref{thm: limsup}. \\

Recall that to compute the squared norm of a holomorphic section
$s$ of $L$ on an open set $U$ with respect to the metric $h^L$ one
takes $|s|^2_{h^L}=|f|^2e^{-2\phi}$, where $s=fe^L$ with $f \in
\mathcal{O}_X(U)$, $e^L$ is a local holomorphic frame for $L$ and
$\phi \in L^1_{loc}$ is a local weight for $h^L$ in $U$. Let $L$ and $h^L$ be fixed. It follows that for a section $s=\{s_i\}$ of $L^j, \quad j \geq 1$ one has, in a trivialization $\theta_i$, $\|s\|_{jh}=|s_i|\exp (-jh)$.

We will need a generalized Poincar\'{e}-Lelong formula for singular metrics:\\

\begin{thm}\label{thm:PL} (\cite{MM},  Theorem 2.3.3) Let $s$ be a meromorphic section of a holomorphic line bundle $L$ over a compact connected complex
manifold $X$ and let $h^L$ be a singular hermitian metric on $L$. Then
\[
\frac{\sqrt{-1}}{2 \pi}\partial \bar{\partial} \log |s|^2_{h^L}= [\mbox{Div}(s)] - c_1(L,h^L).
\]
\end{thm}

We will also need a sequential compactness criterion for currents.
(Recall that a set $\mathcal{T}$ of $(p,p)$-currents of order $0$
on an open domain $\Omega \subset \mathbb{C}^n$ is sequentially
compact in the weak*-topology if every sequence in $\mathcal{T}$
contains a convergent subsequence.) The following can be found in
\cite{Kl}, p. 109:

\begin{prop} A sufficient condition for $\mathcal{T}$ to be sequentially compact in the weak*-topology is that for each relatively compact
subdomain  $G \subset \Omega$ there exists a constant $C > 0$ such
that $|T(\omega)| \leq C \|\omega\|_G$ for every current $T \in
\mathcal{T}$ and every differential form $\omega \in
\mathcal{C}_0^{\infty}(G,
\bigwedge^{n-p,n-p}(\mathbb{C}^n,\mathbb{C}))$.
\end{prop}

Now we will prove the inclusion $N_{amp} \subset \overline{N_{eff}}$ in the toric case:\\

\begin{thm}\label{thm:psef} Let $X$ be a toric manifold with a torus-invariant ample line bundle $L$. Then
$c_1(L,h^L) \in \overline{N_{eff}}$.
\end{thm}

\begin{proof} In a trivializing open neighborhood $U_\iota$ for $L'$ one has the equation
$2\phi=\log(H \circ \theta_\iota^{-1}(x,t))-\log |t|, t \in
\mathbb{C}^*$ for the local weight of $h^L$, where $H$ satisfies the assumption
Theorem \ref{thm:hom-psh}. It follows that  $$\frac{\sqrt{-1}}{2 \pi}\partial
\bar{\partial} \log |s|^2_{h^L}= \frac{\sqrt{-1}}{2 \pi}\partial
\bar{\partial}|f \mid_{U_\iota}|^2 - \frac{\sqrt{-1}}{2
\pi}\partial \bar{\partial}\log(H \circ \theta_\iota^{-1}(x,t)).$$ 
Repeating the
Poincar\'{e}-Lelong formula for every $j \geq 1$ one gets $$ \frac{\sqrt{-1}}{2 \pi}\partial
\bar{\partial}((1/j)\log|s|- h) = [\mbox{Div}(s)] - c_1(L,h^L).$$
It is enough to find a subsequence of the left hand side which tends
to $0$ as a current, so that by Theorem \ref{thm:PL}
$[\mbox{Div}(s)] \to c_1(L,h^L)$ over this subsequence.   We can
assume $t=1$, since the relations between $h$ and $H$ do not
depend on $t$. By Theorem \ref{thm: limsup}, $H(z) =
(\limsup_{j \to \infty}|Q_j|^{1/j})^*(z)$ for some sequence $Q_j$, where each $Q_j$ is a section of $L^j$.
%(The upper star indicates the upper semicontinuous regularization of a function).
The operator $\partial\bar{\partial}$ is weakly continuous, so it
is enough to prove that $(1/j)\log |Q_j| \to \log H$ over a
subsequence, in the sense of currents. Let $u_j :=(1/j)\log |Q_j|$
be such that $\log H= (\limsup_{j \to \infty}u_j)^*$. Invoking the
argument used in the proof of Theorem 2.6.3 in \cite{Kl}, we also have
$\log H = \lim_{k \to \infty}v_k^*$, where $v_k = \sup_{j \geq k}
u_j$. Let $G$ be a relatively compact domain and let
$\omega(x)=w(x)dz_1\wedge d\bar{z}_1\wedge...\wedge dz_n\wedge
d\bar{z}_n$ be an $(n,n)$-form with $w$ smooth and $\mbox{supp }w
\subset G$. For the sequence of $(0,0)$-currents associated with
the functions $v_k^*$ we have $|v_k(\omega)| \leq \sup_G|w| \int
\log H(z) dz_1\wedge d\bar{z}_1\wedge...\wedge dz_n\wedge
d\bar{z}_n =C_G \|\omega\|_G$. Hence there is a weak*-convergent
subsequence of currents, which we also denote by $v_k^*$. By the
Lebesgue dominated convergence theorem, $v_k^* \to \log H$ in the
sense of currents, which completes the proof.
\end{proof}

\end{document}